\documentclass[final]{dmtcs-episciences}

\usepackage[utf8]{inputenc}
\usepackage{subfigure}
\usepackage{url,epsf,amsmath,amssymb,amsfonts,graphicx,tikz}
\usepackage{color}
\usepackage{epsfig}
\usepackage{hyperref}
\usepackage{enumerate}
\usepackage{float}

\usepackage[round]{natbib}

\newtheorem{theorem}{\bf Theorem}[section]
\newtheorem{corollary}[theorem]{\bf Corollary}

\newtheorem{lemma}[theorem]{\bf Lemma}

\newtheorem{proposition}[theorem]{\bf Proposition}
\newtheorem{conjecture}[theorem]{\bf Conjecture}

\author{Tadeja Kraner \v{S}umenjak\affiliationmark{1,2}
  \and Iztok Peterin\affiliationmark{2,3}\thanks{Partially supported by the Ministry of Science of Slovenia under the grant P1-0297.}
  \and Douglas F. Rall\affiliationmark{4}\thanks{This author was supported by the grant ``Internationalisation--a pillar of development of University of Maribor.'' and by a grant from the Simons Foundation (\#209654 to Douglas F. Rall).}
  \and Aleksandra Tepeh\affiliationmark{3,5}\thanks{This author was partially supported by Slovenian research agency ARRS, program no. P1-00383, project no. L1-4292, and Creative Core-FISNM-3330-13-500033.}}
\title[Efficient open domination Cartesian products]{Partitioning the vertex set of $G$ to make $G\,\Box\, H$ an efficient open domination graph}
\affiliation{
  University of Maribor, FKBV, Ho\v{c}e, Slovenia\\
  IMFM, Ljubljana, Slovenia \\
  University of Maribor, FEECS, Maribor, Slovenia \\
  Furman University, Greenville, SC, USA \\
  Faculty of Information Studies,  Novo Mesto, Slovenia}
\received{2015-8-18}
\revised{2016-5-26}
\accepted{2016-5-31}
\begin{document}
\publicationdetails{18}{2016}{3}{10}{1277}
\maketitle

\begin{abstract}
A graph is an efficient open domination graph if there exists a
subset of vertices whose open neighborhoods partition its vertex
set.  We characterize those graphs
$G$ for which the Cartesian product $G\,\Box\, H$ is an efficient open
domination graph when $H$ is a complete graph of order at least 3
or a complete bipartite graph.  The characterization is based on the
existence of a certain type of weak partition of $V(G)$.
For the class of trees when $H$ is complete of order at
least 3, the characterization is constructive.  In addition, a special
type of efficient open domination graph is characterized among
Cartesian products $G \,\Box\, H$ when $H$ is a $5$-cycle or a $4$-cycle.
\end{abstract}

\section{Introduction} \label{section:intro}

\medskip
The domination number of a graph $G$ is a classical invariant
in graph theory. It is the minimum cardinality of a set $S$  of vertices for which the
union of the closed neighborhoods centered in vertices of $S$ is the entire
vertex set of $G$.
Hence, each vertex of $G$ is either in $S$ or is adjacent to a vertex in $S$. In
other words, we can say that vertices of $S$ control each vertex outside of $S$.
A classical question in such a situation is: who controls the
vertices of $S$? One possible solution to this dilemma is total
domination. A set $D\subseteq V(G)$ is a total dominating set
of $G$ if every vertex of $G$ is adjacent to a vertex of $D$.
(Hence, vertices of $D$ are also controlled by $D$.)

A natural question for a graph $G$ is whether we can find a total
dominating set $D$ such that the union of the open neighborhoods of the vertices
in $D$ is $V(G)$ but so that these open neighborhoods also form a partition of
$V(G)$.  The concept has been
presented under the names total perfect codes \cite{CoHaHe}, efficient open
domination \cite{GaSc} and exact transversals \cite{CoHeKe}. In the present
work we follow the terminology of efficient open domination, and we say that a graph
$G$ is an efficient open domination graph if $G$ has a total dominating set $D$ such
that the open neighborhoods of the vertices of $D$ form a partition of $V(G)$.
A similar concept for ordinary domination was first investigated by \cite{biggs} and
\cite{Kra}.  They call a graph  1-perfect if it contains a perfect code, that is,
a set of vertices whose closed neighborhoods partition the vertex set.

The problem of establishing whether a graph $G$ is an efficient open
domination graph is an $NP$-complete problem; see
\cite{GaScSl,McRae}. \cite{GaScSl} gave a recursive
characterization of the class of  efficient open domination trees.
\cite{GaSc} presented various properties of
efficient open domination graphs. The efficient open domination
graphs that are also Cayley graphs were studied by \cite{Tham} and
efficient open domination grid graphs by \cite{CoHeKe,Dej,KlGo}.
Moreover, \cite{AbHamTay} characterized those
direct product graphs that are efficient open domination graphs.

Several graph products have been investigated in the last few
decades and a rich theory involving the structure and recognition of
classes of these graphs has emerged \cite{ImKl}. The most studied
graph products are the Cartesian, strong, direct, and lexicographic.
These four are also called the \emph{standard products}. One
approach to graph products is to deduce properties of a product with
respect to (the same) properties of its factors. See a short
collection of these types involving total domination and perfect
codes in \cite{DoGrSp,Grav,HeRa,Ho,JeKlSp,KlSpZe,KuPeYe,KuPeYe1,Meki,Rall}.
The domination related questions on the Cartesian product seems to
be the most problematic among the standard products. We just mention
Vizing's conjecture, which says that the domination number of a Cartesian product
is at least the product of the domination numbers of the two factors. Settling this
conjecture is one of the most challenging problems in the area of
domination (see the recent survey on Vizing's conjecture
\cite{BrDoKl}). Efficient open domination is no exception, which
could be the reason it has not been studied intensively yet in the
Cartesian product setting. Other than the results on grid graphs mentioned above,
a step forward in this direction was made only recently by \cite{KuPeYe1}
where some special types of Cartesian products were considered. In
the same paper complete descriptions of efficient open domination
graphs among lexicographic and strong products of graph were
given.

The aim of this paper is to show how the problem of finding
efficient open domination graphs among Cartesian products can be
approached by partitioning the vertex set of one factor.  In the next section
we set the context by supplying needed definitions and previous results in this area.  In
Section~\ref{sec:completefactor} we prove that for $r \ge 3$, the graph $G\,\Box\,K_r$
has an efficient open dominating set if and only if
$V(G)$ has a weak partition that satisfies certain properties.  This
provides a way to construct graphs with efficient open dominating sets in this family of Cartesian products.  In addition
we give a structural characterization of the trees $T$ such that $T\,\Box\,K_r$ has an efficient open dominating set.
Section~\ref{sec:diam2} addresses this weak partition approach to graphs of diameter 2.

\medskip

\section{Definitions and previous results} \label{sec:def}
Throughout the article we consider only finite, simple graphs.
For most common graph theory notation and definitions we follow the book by \cite{ImKl}.
In particular, our definitions and notation for open ($N(v)$) and closed ($N[v]$) neighborhoods of a
vertex $v$, for distance ($d_{G}(u,v)$) between a pair of vertices and for the diameter ($\mathrm{diam}(G)$) of
a graph are the same as theirs. The distance $d_{G}(e,v)$
between an edge $e$ and a vertex $v$ in $G$ is the shortest
distance between $v$ and the two end vertices of $e$, while the distance $d_{G}(e_{1},e_{2})$
between edges $e_{1}$ and $e_{2}$ is the shortest
distance between the end vertices of $e_{1}$ and the end vertices of $e_{2}$. In
general, for nonempty subsets $P$ and $Q$ of $V(G)$, the distance $d_{G}(P,Q)$ between them is
the shortest distance between a vertex from $P$ and a vertex from $Q$. A \emph{weak partition}
of a set $X$ is a collection of pairwise disjoint subsets of $X$ whose union is $X$.
We emphasize that, in contrast to a partition, members of a weak partition are allowed to be empty.
The subgraph induced by a subset $S$ of $V(G)$ is denoted by
$\left\langle S\right\rangle$.  A \emph{matching} in $G$ is any (possibly empty) set of
independent edges.  If $r$ is a positive integer, then the vertex
set of each of the graphs $P_r$, $K_r$, and $C_r$ (if $r>2$) will be the interval $[r]$
defined by $[r]=\{1,\ldots,r\}$.

Since this present work concerns total domination on Cartesian products, we include several of the important definitions
here for the sake of completeness.  We say that a vertex $x$ of $G$ \emph{dominates} a vertex $y$ (equivalently, $y$
\emph{is dominated by} $x$) if $y \in N(x)$. A subset $D$ of $V(G)$ is a \emph{total dominating set} of $G$ if each vertex
in $G$ is dominated by at least one vertex in $D$.  The \emph{total domination number} of a graph $G$ is the
minimum cardinality of a total dominating set of $G$ and is denoted by $\gamma_{t}(G)$.
The \emph{Cartesian product}, $G\,\Box\, H$, of graphs $G$ and $H$ is a
graph with $V(G\,\Box\, H)=V(G)\times V(H)$. Two vertices $(g,h)$ and
$(g',h')$ are adjacent in $G\,\Box\, H$ whenever
($gg'\in E(G)$ and $h=h'$) or ($g=g'$ and
$hh'\in E(H)$).  For a fixed $h\in V(H)$ we call
$G^{h}=\{(g,h)\in V(G\,\Box\, H):g\in V(G)\}$ a $G$-\emph{layer} in
$G\,\Box\, H$. Similarly, an $H$-\emph{layer} $^{g}\!H$ for a fixed $g\in V(G)$
is defined as $^{g}\!H=\{(g,h)\in V(G\,\Box\, H):h\in V(H)\}$. Notice
that the subgraph of $G\,\Box\, H$ induced by a $G$-layer or an
$H$-layer is isomorphic to $G$ or $H$, respectively. The map
$p_{G}:V(G\,\Box\, H)\rightarrow V(G)$ defined by $p_{G}((g,h))=g$ is
called a \emph{projection map onto} $G$. Similarly, we define
$p_{H}$ as the \emph{projection map onto} $H$. Projections are
defined as maps between vertices, but frequently it is more
useful to see them as maps between graphs.

A graph $G$ is an \emph{efficient open domination graph} (shortly an \emph{EOD-graph})
if there exists a set $D$, called an \emph{efficient open dominating set} (shortly an \emph{EOD-set}),
for which $\bigcup _{v\in D}N(v)=V(G)$ and $N(u)\cap N(v)=\emptyset$ for every pair $u$
and $v$ of distinct vertices of $D$. Note that two different vertices of an EOD-set are either
adjacent or at distance at least three. It is easy to see that the path $P_{n}$
is an EOD-graph if and only if $n\not\equiv 1 \pmod{4}$, while the cycle $C_{n}$ is an EOD-graph if and
only if $n\equiv 0 \pmod{4}$.
Let $G$ and $H$ be graphs such that $G\,\Box\,H$ is an EOD-graph with an EOD-set $D$. Note
that the projection of an edge in $\left\langle D\right\rangle$ onto
$G$ is either a vertex or an edge. When the projection
of every edge in $\left\langle D\right\rangle$ onto $G$ is an edge,
we say that $D$ is a \emph{parallel EOD-set} with respect to $G$.  A Cartesian product that contains a parallel
EOD-set with respect to one of its factors is called a \emph{parallel EOD-graph}.

\medskip
Among the class of nontrivial Cartesian products several infinite families of EOD-graphs
have been found. In \cite{CoHeKe,KlGo} the authors investigated EOD-graphs
among the grid graphs (that is, Cartesian products of paths). Results from both papers
are merged in the following characterization.

\begin{theorem}
\emph{\cite{CoHeKe,KlGo}} Let $t\geq r\geq 3$.  The grid graph $P_r\,\Box\, P_t$ is an
EOD-graph if and only if $r$ is an even number and
$t\equiv x \pmod{r+1}$ for some $x\in \{1,r-2,r\}$.
\end{theorem}

Some partial results on EOD-graphs in the family of torus
graphs (Cartesian products of cycles) were presented by \cite{Dej}, by characterizing only
those with a parallel EOD-set (there referred to as a parallel
total perfect code).

\begin{theorem}
\emph{\cite{Dej}}\label{torus-parallel} The Cartesian product $C_{r}\,\Box\, C_{t}$ has
a parallel EOD-set if and only if $r$ and $t$ are multiples of four.
\end{theorem}

\cite{KuPeYe1} recently continued with the study of EOD-graphs
among tori and cylinders (Cartesian product of a path and a cycle).

\begin{proposition}
\emph{\cite{KuPeYe1}} Let $t\ge 4$.  The torus $C_4\,\Box\, C_t$ is an EOD-graph if and only if
$t\equiv 0 \pmod{4}$.
\end{proposition}

In addition, they proved that  $C_r\,\Box\, C_t$ is not an
EOD-graph if $r\in \{3,5,6,7\}$ and $t\geq r$. Based on the above observations
they posed the following conjecture.

\begin{conjecture}
\emph{\cite{KuPeYe1}}
Let $r$ and $t$ be integers such that $r \ge 3$ and $t \ge 3$. The torus $C_{r}\,\Box\, C_{t}$
is an EOD-graph if and only if $r \equiv 0 \pmod{4}$ and $t \equiv 0 \pmod{4}$.
\end{conjecture}

The same authors characterized the graphs $G$ for which $G\,\Box\, K_{2}$
is an EOD-graph. In order to do this they introduced the so-called zig-zag graphs, \cite{KuPeYe1}.
Let $G$ be a graph on at
least three vertices and $E'=\{e_{1},\ldots ,e_{k}\}$  a subset of
$E(G)$, where $e_{i}=u_{i}v_{i}$ for every $i\in[k]$, with the following properties:

\begin{itemize}
\item[$(i)$] $N(u_{i})\cap N(v_{i})=\emptyset $;

\item[$(ii)$] $d_{G}(e_{i},e_{j})\geq 2$ for $1 \le i<j\le k$;

\item[$(iii)$] for every $x\in V(G)-\{u_{i},v_{i}:i\in [k]\}$ there exist
unique $j$ and $\ell $, $j\neq \ell $, such that $d_{G}(x,e_{j})=d_{G}(x,e_{\ell })=1$;

\item[$(iv)$] for every sequence $e_{i_{1}},\ldots ,e_{i_{j}}$ of distinct edges
with $j>2$ and with \\
$d_{G}(e_{i_{_{\ell}}},e_{i_{_{\ell +1 \pmod{j}}}})=2$
for $\ell \in \{1,\ldots ,j\}$, $j$ must be an even number.
\end{itemize}

We call $E'$ a \emph{zig-zag set} of $G$ and, if there exists a
zig-zag set in $G$, we call $G$ a \emph{zig-zag graph}.

\begin{theorem}\label{zigzagEOD}
\emph{\cite{KuPeYe1}}
If $G$ is a zig-zag graph, then $G\,\Box\, K_{2}$ is an EOD-graph.
\end{theorem}

Not all EOD-graphs among $G\,\Box\, K_{2}$ are given by the above
theorem. Kuziak et al. observed that for a description of all
EOD-graphs among Cartesian products of graphs with $K_2$, a certain
combination of zig-zag graphs and 1-perfect graphs is needed (see
\cite{KuPeYe1} for details).

One can observe that for $r>2$, every
EOD-set in $G\,\Box\, K_{r}$ is a parallel EOD-set with respect to $G$.
Namely, if an edge induced by two vertices of a vertex subset $A$ of
$G\,\Box\, K_{r}$ projects to a single vertex $g\in V(G)$,
then the layer $^{g}\!K_r$ contains a vertex that is dominated more than once by $A$.
This observation led to the idea of how to approach the problem of finding EOD-graphs among
$G\,\Box\, K_{r}$ for $r>2$.  This is presented in the next section.

\section{$G\,\Box\, K_{r}$ for $r>2$} \label{sec:completefactor}

In order to obtain a characterization of EOD-graphs
among $G\,\Box\, K_{r}$, $r>2$, we introduce a new concept,
based on a weak partition of the vertex set of $G$. As we will see in later sections, a
modification of this concept can be used for the construction of EOD-graphs that are Cartesian products
$G \,\Box\, H$ where $H$ belongs  to several other special classes of graphs.

Let $r$ be an integer larger than 1.  We call a graph $G$ a
\emph{$K_{r}$-amenable graph} if there exists a weak partition $\{V_{0},V_{1},\ldots
,V_{r}\}$ of $V(G)$, such that

\begin{itemize}
\item[\textrm{(A)}] if $x\in V_{0}$, then $|N(x)\cap V_{i}|=1$ for every
$i\in [r]$,

\item[\textrm{(B)}] $\left\langle V_{i}\right\rangle$ is a matching in $G$ for
every $i\in [r]$,

\item[\textrm{(C)}] $\left\langle V_{1}\cup \cdots \cup V_{r}\right\rangle$
is a matching in $G$.
\end{itemize}

For the sake of clarity in the above definition we emphasize that the induced
subgraphs $\left\langle V_{i}\right\rangle$ and
$\left\langle V_{1}\cup \cdots \cup V_{r}\right\rangle$
do not contain any edges other than those in their perfect matchings.

We first prove that $K_{2}$-amenable graphs do not differ from zig-zag graphs.

\begin{theorem}
A graph $G$ is a $K_{2}$-amenable
graph if and only if $G$ is a zig-zag graph.
\end{theorem}

\begin{proof} Let $G$ be a $K_{2}$-amenable graph
with a weak partition $\{V_{0},V_{1},V_{2}\}$ of $V(G)$ that satisfies
conditions (A), (B) and (C). We will show
that $E'=\left\langle V_{1}\cup V_{2}\right\rangle $ is a
zig-zag set of $G$ by demonstrating that conditions $(i)-(iv)$ hold. Since $\left\langle V_{1}\cup V_{2}\right\rangle$,
$\left\langle V_{1}\right\rangle $ and $\left\langle
V_{2}\right\rangle $ are matchings, $E'$ is a set of edges
$\{e_{1},\ldots ,e_{k}\}$.
By the same argument we derive that $d_{G}(e_{i},e_{j})\geq 2$ for $i \neq j$, and thus $(ii)$ holds.
Let $e_{i}=u_{i}v_{i}$ for every $i\in [k]$.  If
$x\in N(u_{i})\cap N(v_{i})$ for some $i\in [k]$, then $x\in V_{0}$ by
the matching argument again. But this contradicts condition
(A) since $|N(x)\cap V_{i}|\geq 2$ in this case. Hence $(i)$ also
holds. If $x\in V(G)-\{u_{i},v_{i}:i\in [k]\}$,
then $x\in V_{0}$. By (A) we have that $|N(x)\cap V_{i}|=1$
for every $i\in \{1,2\}$, which implies the existence of exactly two
different edges $e_{j}$ and $e_{\ell }$ of $E'$ with
$d_{G}(x,e_{j})=d_{G}(x,e_{\ell })=1$.  This proves $(iii)$. To prove
$(iv)$, let $e_{i_{1}},e_{i_{2}},\ldots ,e_{i_{j}}$, $j>2$, be a
sequence of distinct edges with $d_{G}(e_{i_{_{\ell }}},e_{i_{_{\ell
+1 \pmod{j}}}})=2$ for $\ell \in [j]$. In
addition, let $x_{\ell }$ be a common neighbor of $e_{i_{_{\ell }}}$
and $e_{i_{_{\ell +1\pmod{j}}}}$. As before, $x_{\ell }\in V_{0}$
for every $\ell \in [j]$. Without loss of
generality, suppose the end-vertices of the edge $e_{i_{1}}$ belong to
$V_1$. By condition (A) for the vertex $x_1$, the
end-vertices of $e_{i_{2}}$ belong to $V_2$. The same argument for
the vertex $x_2$ implies that the end-vertices of $e_{i_{3}}$ belong to
$V_1$. Continuing this way, we get a zig-zag pattern for the
end-vertices of $e_{i_{1}},e_{i_{2}},\ldots ,e_{i_{j}}$. If $j$ is
an odd number, then the end-vertices of $e_{i_{1}}$ and $e_{i_j}$
are both in $V_{1}$, which gives a contradiction with condition
(A) for the vertex $x_j$. Thus $j$ is an even number and
$(iv)$ holds as well.

Now let $G$ be a zig-zag graph with a zig-zag set $E^{\prime
}=\{e_{1},\ldots ,e_{k}\}$ where $e_{i}=u_{i}v_{i}$. We set
$V_{0}=V(G)-\{u_{i},v_{i}:i\in [k]\}$.  Observe that $E'$ can be partitioned as
$E'=E_1\cup \cdots \cup E_t$ such that for each $i\in [t]$, the following holds.
The set $E_i$ is  a maximal
set of edges such that between any two distinct edges $e_{j}$ and
$e_n$ from $E_i$ there exists a sequence
$e_{j}=e_{j_0},e_{j_{1}},\ldots ,e_{j_{\ell}}=e_n$, $\ell \geq 1$, of
distinct edges where the distance between two consecutive edges
in this sequence is $2$.  Such a sequence is called a {\em 2-step
sequence of length} $\ell$.

Observe that there
exists a partition of $E'=E_1\cup \cdots \cup E_t$, such
that $E_i$, for every $i\in [t]$, consists of a maximal
set of edges such that between any two distinct edges $e_{j}$ and
$e_k$ from $E_i$ there exists a sequence
$e_{j}=e_{j_0},e_{j_{1}},\ldots ,e_{j_{\ell}}=e_k$, $\ell \geq 1$, of
distinct edges such that the distance between two consecutive edges
in this sequence is $2$ (we call such sequence a {\em 2-step
sequence of length} $\ell$). Now, in $E_i$ fix an arbitrary edge $e$. For
an arbitrary edge $f$ in $E_i$ there exists a 2-step sequence
between $e$ and $f$. Property $(iv)$ implies that the lengths of all different
2-step sequences between $e$ and $f$ are of the same parity. Thus,
edges of $E_i$ can be partitioned into two sets $E_i^{1}$ and
$E_i^{2}$. The set $E_i^{1}$ consists of $e$ and all edges $f$ for
which the length of a 2-step sequence between $e$ and $f$ is even,
and $E_i^{2}=E_i-E_i^{1}$. For every $i\in [t]$
let $V_i^{1}$ denote the set of end-vertices of edges in $E_i^{1}$,
and $V_i^{2}$ the set of end-vertices of edges in $E_i^{2}$.
Finally, let $V_{1}=V_1^{1}\cup \cdots \cup V_t^{1}$ and
$V_{2}=V_1^{2}\cup \cdots \cup V_t^{2}$.

We will show that $\{V_{0},V_{1},V_{2}\}$ is a weak partition of $V(G)$
satisfying conditions (A), (B) and (C). Properties (B) and
(C) clearly follow, since $d_{G}(e_{i},e_{j})\geq 2$ for every
pair $e_{i},e_{j}\in E'$. To prove (A) let $x\in
V_{0}$. By $(iii)$ there exist exactly two different edges
$e_{p},e_{r}\in E'$ such that
$d_{G}(x,e_{p})=1=d_{G}(x,e_{r})$. Note that $e_{p}$ and $e_{r}$
belong to the same $E_i$ in the partition of $E'$. Recall
that we have fixed the edge $e\in E_i$. If a 2-step sequence between
$e$ and $e_{p}$ and a 2-step sequence between $e$ and $e_{r}$ have
the same parity, then we obtain a contradiction with $(iv)$. Hence,
end-vertices of one edge, say $e_{p}$, belong to $V_{1}$, and
end-vertices of $e_{r}$ belong to $V_{2}$. Since, in addition,
$N(u_{i})\cap N(v_{i})=\emptyset $, by $(i)$ for every $i$
we have $|N(x)\cap V_{1}|=1=|N(x)\cap V_{2}|$ and
condition (A) holds.
\end{proof}

\begin{theorem}
\label{complete} Let $r$ be a positive integer such that $r>2$  and let $G$ be a graph. The Cartesian product
$G\,\Box\, K_{r}$ is an EOD-graph if and only if  $G$ is a $K_{r}$-amenable graph.
\end{theorem}

\begin{proof} Let $G$ be a $K_{r}$-amenable graph with
corresponding weak partition $\{V_{0},\ldots,V_{r}\}$ of $V(G)$. We
define a subset $D$ of $V(G\,\Box\, K_{r})$ by
$D=\{(g,i): i \in [r] \mbox{ and } g \in V_i\}$.  It follows that
$D$ contains at most one vertex from each $K_r$-layer.
To prove that $G\,\Box\, K_{r}$ is an EOD-graph
we will show that every vertex of $G\,\Box\, K_{r}$ is dominated by exactly one vertex of $D$.
Let $i \in [r]$ and let $g \in V(G)$.
First, assume that $g\in V_{0}$. By (A), the vertex $g$ has a
unique neighbor $x_{i}$ in $V_{i}$. Consequently,
$(g,i)$ is adjacent to $(x_{i},i)$ and $(x_{i},i)\in D$. Moreover, by the uniqueness of
$x_{i}$, no other vertex of $D$ dominates $(g,i)$. Now assume that $g\in V_{i}$.
Since  $\left\langle V_{i}\right\rangle $ is a perfect matching, $g$ has a unique
neighbor $g'$ in $V_i$.  It follows that $(g',i)\in D$ and that
$(g',i)$ is the only neighbor of $(g,i)$ in $D$. Finally, assume that
$g \in V_j$ for some $j \in [r]$ such that $j\neq i$.  By the definition of $D$ this implies that
$\{(g,j)\} = D \cap\,^{g}K_{r}$.  In addition, since (B) and (C) hold, $(g,i)$
has no neighbor in $G^i \cap D$.  The result is that $(g,i)$ is dominated by
exactly one vertex, namely $(g,j)$, of $D$.  Consequently, $D$ is an EOD-set of $G \,\Box\, K_r$
and $G\,\Box\, K_{r}$ is an EOD-graph.

To prove the converse, suppose that $G\,\Box\, K_{r}$ is an EOD-graph with an
EOD-set $D$. For $i\in [r]$ let $V_i=\{v\in V(G): (v,i)\in D\}$,
and let $V_{0}=V(G)-(V_{1}\cup \cdots \cup V_{r})$.
As we observed in Section~\ref{sec:def}, $D$ is necessarily parallel with respect to $G$.
This means that every $^{v}K_{r}$ contains at most one vertex of $D$, and we thus infer that
$\{V_{0}, V_1,\ldots ,V_{r}\}$ is a weak partition of $V(G)$.
 We prove that conditions (A), (B), and (C) of the definition of $K_r$-amenable hold.
If condition (A) is not satisfied, then there exist $x\in V_{0}$ and $i\in [r]$, such
that $|N(x)\cap V_{i}|=0$ or $|N(x)\cap V_{i}|>1$. In the first case
$(x,i)$ is not dominated by any vertex of $D$, and in the second case $(x,i)$
is dominated by more than one vertex of $D$. Both cases are in contradiction
with the assumption that $D$ is an EOD-set of $G\,\Box\, K_{r}$.  Hence, the weak partition
$\{V_{0},V_1,\ldots ,V_{r}\}$ satisfies property (A).  Let $i \in [r]$ and let $g \in V_i$.
Since $|D \cap\, ^gK_r|\le 1$ and $(g,i)$ has exactly one neighbor in $D$, it
follows that $|N(g) \cap (V_1 \cup \cdots \cup V_r)|= 1 = |N(g) \cap V_i|$.
Hence, both (B) and (C) hold.  Therefore, $G$ is a $K_{r}$-amenable graph.
\end{proof}

Let $r$ be an integer larger than 1.  In the rest of this section we present a recursive
description of the family of all $K_{r}$-amenable trees. The following construction generalizes the
construction of zig-zag trees (that is, $K_{2}$-amenable trees) from \cite{KuPeYe1}. We will
denote by $K_{1,r}^{+}$ the tree of order $2r+1$ obtained from the star $K_{1,r}$
by subdividing each edge exactly once.  It is clear that $K_{1,r}^{+}$ is a $K_{r}$-amenable
tree, and the corresponding partition of $V(K_{1,r}^{+})$ is unique up to a permutation of
$[r]$.  We now define an infinite family $\mathcal{T}_{r}$ of trees.  Each member of
$\mathcal{T}_{r}$ will have a weak partition $\{V_0, V_1, \ldots, V_r\}$ of its vertex set associated with it.

Suppose that $T'$ is a tree of order $n$ such that $\{V'_{0}, V'_{1},\ldots,V'_{r}\}$
is a weak partition of $V(T')$ and that $T''$ is a tree of order $m$ such that
$\{V''_{0}, V''_{1}, \ldots ,V''_{r}\}$ is a weak partition of
$V(T'')$.

We say that a tree $T$ of order $n+m-2$ is obtained from $T'$ and
$T''$ by a {\bf Type-a} construction if $T$ is isomorphic to the tree
formed by choosing any $i \in [r]$, any edge $u'_{i}v'_{i}$ in
$\left\langle V'_{i}\right\rangle$, any edge $u''_{i}v''_{i}$
in $\left\langle V''_{i}\right\rangle$ and then
identifying the vertices $u'_{i}$ with $u''_{i}$
(now called $u_{i}$) and $v'_{i}$ with $v''_{i}$ (now called $v_{i}$)
to obtain the edge $u_{i}v_{i}$ in $T$.  The associated weak partition $\{V_0, V_1, \ldots,V_r\}$
of $V(T)$ is defined by $V_j=V'_{j} \cup V''_{j}$ if $j \neq i$, and
$V_i=(V'_{i} \cup V''_{i} \cup \{u_i,v_i\})-\{u'_{i},v'_{i}, u''_{i}, v''_{i}\}$.

A tree $S$ of order $n+m$ is obtained from $T'$ and
$T''$ by a {\bf Type-b} construction if $S$ is isomorphic to the tree
formed from the union of $T'$ and $T''$ by adding an edge
$xy$ for some $x\in V'_{0}$ and some $y\in V''_{0}$.
The associated weak partition $\{V_0, V_1, \ldots,V_r\}$ of $V(S)$ is given by
$V_i= V'_{i} \cup V''_{i}$ for $0 \le i \le r$.

The family $\mathcal{T}_{r}$ is defined recursively as follows. A tree $T$ belongs to
$\mathcal{T}_{r}$  if and only if $T=K_{1,r}^{+}$
with its partition as indicated above or $T$ can be obtained from smaller trees in
$\mathcal{T}_{r}$ by a finite sequence of  Type-a or Type-b constructions.

\begin{theorem}
\label{zigzagtree} Let $r$ be an integer such that $r\ge 2$. The path of order 2 is $K_{r}$-amenable.
If $T$ is a tree of order more than 2, then $T$ is a $K_r$-amenable graph if and only if
$T\in \mathcal{T}_{r}$.
\end{theorem}

\begin{proof}
Let $r$ be an integer such that $r\ge 2$.
For the path of order 2, let $V_1=V(P_2)$, $V_0=\emptyset=V_i$ for $2 \le i \le r$.
This weak partition $\{V_0,V_1,\ldots,V_r\}$ satisfies the definition showing that $P_2$ is a
$K_r$-amenable graph.  For the remainder of this proof we assume that all trees under consideration
have order at least 3.
As noted above, the tree $K_{1,r}^{+}$ is a $K_{r}$-amenable tree.  One can conclude directly
from the definitions that if $T'$ and $T''$ are both $K_{r}$-amenable
trees, then a tree obtained from $T'$ and $T''$
by a Type-a or a Type-b construction is also a $K_{r}$-amenable graph.
Thus, it follows by induction (on the number of Type-a and Type-b
constructions) that every member of $\mathcal{T}_{r}$ is a $K_{r}$-amenable
graph.

Conversely, let $T$ be a $K_{r}$-amenable tree of order at least 3 with a corresponding
weak partition $\{V_{0},V_1,\ldots ,V_{r}\}$ and let $k=|V_0|$.
Since $T$ has order at least 3, it follows from the definition that $k\ge 1$.   We use induction on $k$
to show that $T\in \mathcal{T}_{r}$. Let $k=1$ and $V_{0}=\{v\}$. By property (A)
$\deg(v)=r$; let $N(v)=\{u_{1},\ldots ,u_{r}\}$  where $u_{i}\in V_{i}$.
By (B) every $u_{i}$ has a unique neighbor $w_{i}$ in $V_{i}$ and by
(C) $u_{i}$ and $w_{i}$ have no neighbors in $V_{j}$ for $j\neq i$. Moreover, $u_{i}$ and
$w_{i}$ have no additional neighbors in $V_{0}$ since $k=1$. Thus,
$T$ is isomorphic to $K_{1,r}^{+}$ and hence $T\in \mathcal{T}_{r}$.

Now suppose that $k>1$. Note that every vertex in $V_0$ has degree at least $r$.
If there exists $v\in V_{0}$ with $\deg (v)>r$, then there exists
$w \in V_{0}\cap N(v)$. Let $T'$ be the component of $T-vw$ that
contains $v$ and let $T''$ be the component that contains $w$.  For
$0 \le i \le r$, let $V_{i}'=V_{i}\cap V(T')$ and let
$V_{i}''=V_{i}\cap V(T'')$. The resulting weak partitions of
$V(T')$ and $V(T'')$
clearly satisfy properties (A), (B) and (C), and furthermore $|V_{0}'|<k$
and $|V_{0}''|<k$. By the induction hypothesis both $T'$ and
$T''$ belong to $\mathcal{T}_{r}$.  Since $T$ is obtained from
$T'$ and $T''$ by a Type-b construction, it follows
that $T\in \mathcal{T}_{r}$.

Now, suppose that all vertices of $V_{0}$ are of degree $r$ (and hence
$\left\langle V_{0}\right\rangle$ contains no edges). Choose $u$ and $v$
from $V_{0}$ with the property that $d_{T}(u,v)$ is minimum among all
different pairs of vertices from $V_{0}$. Clearly, $2\leq d_{T}(u,v)\leq 3$.
Let $w$ be the neighbor of $u$ on the shortest $u,v$-path in $T$. Without
loss of generality we may assume that $w\in V_{1}$. By (B), $w$ has a
unique neighbor, say $w'$, in $V_{1}$. The forest $T-uw$ has
two connected components.  The component that contains $u$ is denoted by $T_{u}$ and the one that
contains $v$ is denoted by $T''$. Let $T'$ be the tree obtained from $T_u$
by adding vertices $t$ and $t'$ and adding edges $ut$ and $tt'$.
Let $V_i''=V_i\cap V(T'')$ for  $0 \le i \le r$, let
$V_1'=(V_1\cap V(T_u))\cup \{t,t'\}$, and let
$V_i'=V_i\cap V(T_u)$ for $i=0$ and  $2 \le i \le r$.
Properties (A), (B) and (C) clearly hold for the above defined weak partitions
of $V(T')$ and $V(T'')$. Thus, $T'$ and $T''$ are $K_{r}$-amenable
trees.  By the induction hypothesis, they are also in $\mathcal{T}_{r}$.
Note that $T$ is isomorphic to the tree obtained from $T'$ and $T''$ by
a Type-a construction that identifies $t$ with $w$, and $t'$ with $w'$.
Consequently, $T \in \mathcal{T}_{r}$.  
\end{proof}

This theorem together with Theorem~\ref{complete} combine to give us the following
characterization of those trees $T$ such that $T\,\Box\,K_r$ is an EOD-graph for $r\ge3$.
\begin{corollary}
Let $r$ be a positive integer larger than 2 and let $T$ be a tree.  The Cartesian product  $T\,\Box\,K_r$
is an EOD-graph if and only if $T=P_2$ or $T\in \mathcal{T}_{r}$.
\end{corollary}

\section{$G\,\Box\, H$ with diam$(H)=2$} \label{sec:diam2}

In this section we consider Cartesian products of graphs where (at
least) one factor has diameter $2$.
Motivation for the study of such graphs
arises from the previous section. An EOD-set of $G\,\Box\, H$ that is parallel
with respect to $G$ when $\mathrm{diam}(H)=2$ shares an important property
with such a set in $G \,\Box\, K_r$ for $r \ge 3$.  This is given in the
following lemma.

\begin{lemma} \label{lem:atmosttwo}
Let $H$ be a graph of diameter 2 and let $G$ be a graph such that $G\,\Box\, H$ has
an EOD-set $D$.  For every vertex $g$ in $G$, $|D \cap\,^{g}H| \le 2$.
If in addition $D$ is parallel with respect to $G$, then $|D \cap\,^{g}H| \le 1$ for every $g\in V(G)$.
If $|D \cap\,^{g}H| = 2$, then the two distinct vertices in $D \cap\,^{g}H$ are adjacent.
\end{lemma}
\begin{proof} Assume that $D$ is an EOD-set of $G \,\Box\,
H$ and suppose that $(g,u)$ and $(g,v)$ are distinct vertices in
$D$.  The graph $H$ has diameter 2, and this implies that $uv\in E(H)$ or $u$ and $v$
have a common neighbor $w$ in $H$.  Since every vertex in $^{g}H$ is
dominated exactly once by $D$, we infer that $(g,u)$ and $(g,v)$ are
adjacent, and $|D \cap\,^{g}H| \le 2$. It follows immediately that
if $D$ is parallel with respect to $G$, then no $H$-layer can contain
two members of $D$. \hfill $\,\Box\,$
\end{proof}

As we will see, finding an appropriate weak partition of vertices in $G$ will be useful
in the characterization of (parallel) EOD-graphs among Cartesian products
$G\,\Box\, H$ where $\mathrm{diam}(H)=2$.
First we show that the Cartesian product of a graph of diameter $2$ and a tree
on at least three vertices does not admit a parallel EOD-set with respect to the tree.

\begin{theorem}
\label{diameter} Let $H$ be a graph with $\mathrm{diam}(H)=2$ and let $T$ be
a  tree. If $T$ is different than $K_{2}$, then $T\,\Box\, H$ does not contain a parallel
EOD-set with respect to $T$.
\end{theorem}

\begin{proof} Let $H$ be a graph with $\mathrm{diam}(H)=2$.
Suppose, in order to obtain a contradiction, that there exists a tree
$T$ different than $K_{2}$, such that $T\,\Box\, H$ admits
a parallel EOD-set $D$ with respect to $T$.

First, we claim that $T^{h}\cap D= \emptyset$
for every non-universal vertex $h$ in $H$. If this does not hold, then there exist vertices
$(u_{0},h),(v_{0},h)\in D$ which are adjacent in $T\,\Box\, H$. Since $h$ is not universal in $H$,
there is $h' \in V(H)$ such that $d_H(h,h')=2$. Observe that
$(u_{0},h')$ and $(v_{0},h')$ are not dominated by $(u_{0},h)$ and $(v_{0},h)$.
Moreover, they are not dominated by any vertex in $^{u_0}\!H$ and
$^{v_0}\!H$ (since $\mathrm{diam}(H)=2$ we have that $|^{x}H\cap D|\leq 1$ for every
$x\in V(T)$ by Lemma \ref{lem:atmosttwo}).
Therefore, there exists a neighbor $u_{1}$ of $u_{0}$ and a neighbor $v_{1}$
of $v_{0}$, such that $(u_{0},h')$ is dominated by $(u_{1},h')\in D$ and $(v_{0},h')$
is dominated by $(v_{1},h')\in D$. Moreover, since $(u_{1},h'),(v_{1},h')\in D$, there exist
$(u_{2},h'),(v_{2},h')\in D$, where $u_{2}u_{1},v_{2}v_{1}\in E(T)$. To dominate
vertices $(u_{2},h)$ and $(v_{2},h)$, there must exist $(u_{3},h),(v_{3},h)\in D$
where $u_{3}u_{2},v_{3}v_{2}\in E(T)$. Continuing in
this way we obtain a two-way infinite walk $\ldots
u_{2}u_{1}u_{0}v_{0}v_{1}v_{2}\ldots$ in $T$. Since $T$ is a tree, all vertices of this walk are pairwise different.
But this is in contradiction with $T$ being finite, and the claim is proved.

We infer that $H$ has to contain universal vertices and that the projection of
every edge in $\left\langle D \right\rangle$
onto $H$ is a universal vertex. Now, let $h$ be a universal vertex of $H$,
such that $T^{h}\cap D\neq \emptyset $ and let $(u,h),(v,h)\in D$ be
adjacent vertices. Together they dominate all vertices of $^{u}H$ and $^{v}H$.
There also exists a non-universal vertex $h'$ in $H$ because
$\mathrm{diam}(H)=2$. Since $T$ is different than $K_{2}$, at least one of $u$ and $v$, say $u$, has
a neighbor $w$ in $T$. Note that $(w,h)$ is dominated by $(u,h)$, and
$(w,h')$ is not dominated by $(u,h)$ nor $(v,h)$. Since $
T^{h'}\cap D=\emptyset $, there exists another universal vertex
$h_{1}\in V(H)$, such that $(w,h')$ is dominated by $(w,h_{1})\in D$.
This yields a final contradiction, since $(w,h)$ is dominated by both
$(u,h)$ and $(w,h_{1})$ from $D$, which is not possible in an EOD-set
$D$. 
\end{proof}

\subsection{$G\,\Box\, K_{m,n}$\label{sec:completebipartite}}

In this subsection we give a necessary and sufficient condition on a graph $G$
such that $G\,\Box\, K_{m,n}$ is an EOD-graph for $1 \le m\le n$.  The condition will
be the existence of a weak partition of $V(G)$ that satisfies very specific
requirements. While it may not be easy to determine whether a given
graph $G$ has such a weak partition, the requirements of the weak partition will make it
straightforward to construct graphs $G$ such that $G \,\Box\, K_{m,n}$ is an EOD-graph.

Since $K_{m,n}$ has diameter 2 and we are not requiring the EOD-set
of $G \,\Box\, K_{m,n}$ to be parallel with respect to $G$, we will refer often to
Lemma~\ref{lem:atmosttwo}.  For ease of explanation we assume throughout this
subsection that $1 \le m \le n$ and that $K_{m,n}$ has partite sets $A$ and $B$ given
by $A=\{1,\ldots,m\}$ and $B=\{m+1,\ldots,m+n\}$.  With this notation we let
$\mathcal{C}_{m,n}$ be a weak partition of $V(G)$ containing $mn+m+n+1$ parts indexed
as follows:
\begin{itemize}
\item $V_0,V_1,\ldots, V_m, V_{m+1},\ldots, V_{m+n}$; and
\item $V_{[i,m+j]}$ for $1 \le i \le m$ and $1 \le j \le n$.
\end{itemize}

We will say that $\mathcal{C}_{m,n}$ is \emph{$K_{m,n}$-amenable} if it is
a weak partition satisfying the following conditions.
\begin{itemize}
\item[\textrm{(I)}] For $1 \le i \le m+n$, the induced subgraph $\left\langle V_{i}\right\rangle$
is a matching.
\item[\textrm{(II)}] For $1 \le i \le m$ and $m+1 \le j \le m+n$, $\left\langle V_{i}\cup V_j\right\rangle$
is a matching.
\item[\textrm{(III)}] For $1 \le i<j\le m$ or $m+1\le i < j \le m+n$, each $x$ in  $V_i$ has
exactly one neighbor in $V_j$ and each $y$ in  $V_j$ has exactly one neighbor in $V_i$.
\item[\textrm{(IV)}] If $x \in V_{[i,m+j]}$ for some $1 \le i \le m$ and some $1 \le j \le n$,
then $N(x) \subseteq V_0$.
\item[\textrm{(V)}] If $x \in V_0$, then $|N(x) \cap \left( \cup_{1 \le j\le n}V_{[i,m+j]} \cup V_i\right)|=1$ for $1 \le i \le m$,
and \\$|N(x) \cap \left( \cup_{1 \le i\le m}V_{[i,m+j]} \cup V_{m+j}\right)|=1$ for $1 \le j \le n$.
\end{itemize}

A graph $G$ will be called \emph{$K_{m,n}$-amenable} if $V(G)$ has a weak partition that is
$K_{m,n}$-amenable.  With this definition we are now able to give a constructive characterization
of those graphs $G$ such that $G \,\Box\, K_{m,n}$ is an EOD-graph.

\begin{theorem} \label{thm:completebiparitite}
Let $m$ and $n$ be positive integers such that $m \le n$ and let $G$ be a graph.  The Cartesian product
$G \,\Box\, K_{m,n}$ is an EOD-graph if and only if $G$ is $K_{m,n}$-amenable.
\end{theorem}

\begin{proof} Assume that $G$ is $K_{m,n}$-amenable and that $\mathcal{C}_{m,n}$ is
a weak partition of $V(G)$ indexed as above and satisfying the conditions (I)-(V) in the definition above.
We define a subset $D$ of $V(G \,\Box\, K_{m,n})$ by specifying its intersection with each $K_{m,n}$-layer.
If $r$ is an integer such that $1 \le r \le m+n$ and $g \in V_r$, then $D \cap\,^{g}K_{m,n}=\{(g,r)\}$.
If $r$ and $s$ are integers with $1 \le r \le m$ and $1 \le s \le n$ such that $g \in V_{[r,m+s]}$,
then $D \cap\,^{g}K_{m,n}=\{(g,r),(g,m+s)\}$.  Finally, if $g \in V_0$, then $D \cap\,^{g}K_{m,n}=\emptyset$.
Since $\mathcal{C}_{m,n}$ is a weak partition, the set $D$ is well-defined.  We now show that $D$ is an EOD-set
of $G \,\Box\, K_{m,n}$ by showing that each vertex of  $G \,\Box\, K_{m,n}$ has exactly one neighbor in $D$.

Let $(x,t)$ be an arbitrary vertex in $G \,\Box\, K_{m,n}$.  Assume $x \in V_0$.  Suppose first that $1 \le t \le m$.
By (V) there exists  $y \in V(G)$ such that $\{y\}= N(x) \cap \left( \cup_{1 \le j\le n}V_{[t,m+j]} \cup V_t\right)$.
This implies that $(y,t) \in D$ and that $(y,t)$ dominates $(x,t)$.  Furthermore, it follows from (V) and
$D \cap\,^{x}K_{m,n}=\emptyset$ that $(y,t)$ is the only neighbor of $(x,t)$ that belongs to $D$. The case
$m+1 \le t\le m+n$ is similar.  Assume next that $x \in V_{[r,m+s]}$ for some $r$ and $s$ such that
$1 \le r \le m$ and $1 \le s \le n$.  By the definition of $D$ we get that both $(x,r)$ and $(x,m+s)$ belong
to $D$.  Exactly one of these is adjacent to $(x,t)$.  Combining this with property (IV) it follows that
$(x,t)$ has exactly one neighbor in $D$.  Finally, assume that $x \in V_r$ for some $r$ with $1 \le r \le m$.
(The case $m+1 \le r \le m+n$ is similar.)  This means that $(x,r) \in D$ and $|D \cap\,^{x}K_{m,n}|=1$.
 There are three subcases to consider, namely (i) $m+1 \le t \le m+n$, (ii) $t=r$, and (iii) $t \neq r$
 but $1 \le t \le m$. If $m+1 \le t \le m+n$, then $(x,r)$ dominates $(x,t)$ (from within the
 layer $^{x}K_{m,n}$).  From (I), (II) and (IV) we see that $(x,t)$ is not adjacent to any vertex in
 $D \cap G^t$.  Thus, in subcase (i) $(x,t)$ has a unique neighbor in $D$.  Assume that $t=r$.  By (I)
 there is a unique $y \in V_r\cap N(x)$.  By definition $(y,r)\in D$ and thus $(x,r)$ is dominated by $D$.
 Properties (I) and (IV) together imply that $(x,t)$ has no other neighbor in $D$.  Finally, assume
 subcase (iii) holds.  By (III) there exists a unique $z \in V_t \cap N(x)$.  Now $(z,t)\in D$ and
 $(x,t)$ is dominated by $(z,t)$.  Consequently, by (IV) it follows that $(z,t)$ is the only vertex
 in $D$ that dominates $(x,t)$.  We have shown that $D$ is an EOD-set of $G \,\Box\, K_{m,n}$.

Conversely, suppose that $S$ is an EOD-set of $G\,\Box\, K_{m,n}$.  Since $K_{m,n}$ has diameter 2,
 we apply Lemma~\ref{lem:atmosttwo} and conclude that $|S \cap\,^{g}K_{m,n}| \le 2$ for every vertex
 $g$ in $G$.  We produce a weak partition $\mathcal{C}$ of $V(G)$ as follows.  The sets in $\mathcal{C}$ are
  those in the following specifications.  Note that some of these subsets might be empty.
 \begin{itemize}
 \item $V_0=\{ x \in V(G): S \cap\,^xK_{m,n}=\emptyset\}$,
 \item $V_i=\{ x \in V(G): S \cap\,^xK_{m,n}=\{(x,i)\}\}$ for $1 \le i \le m+n$,
 \item  $V_{[i,m+j]}=\{ x \in V(G): S \cap\,^xK_{m,n}=\{(x,i),(x,m+j)\}\}$ for $1 \le i \le m$ and $1 \le j \le n$.
 \end{itemize}

The verification that $\mathcal{C}$ is $K_{m,n}$-amenable (that is,
it satisfies properties (I)-(V)) follows directly from the
assumption that $S$ is an EOD-set of $G \,\Box\, K_{m,n}$ and is left to
the reader. 
\end{proof}

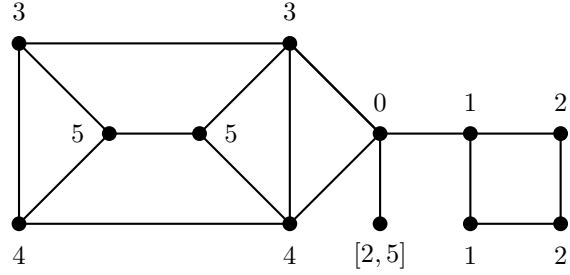
\begin{figure}[ht!]
\begin{center}
\begin{tikzpicture}[scale=0.6,style=thick]
\def\vr{4pt}
\path (6,0) coordinate (a); \path (4,-2) coordinate (b);
\path (-2,-2) coordinate (c); \path (0,0) coordinate (d);
\path (2,0) coordinate (e); \path (4,2) coordinate (f);
\path (-2,2) coordinate (g); \path (8,0) coordinate (u);
\path (8,-2) coordinate (v); \path (10,0) coordinate (w);
\path (10,-2) coordinate (x); \path (6,-2) coordinate (y);
\draw (a) -- (b); \draw (a) -- (u);
\draw (a) -- (y); \draw (b) -- (c);
\draw (b) -- (e); \draw (b) -- (f);
\draw (c) -- (d); \draw (c) -- (g);
\draw (a) -- (f);
\draw (d) -- (e); \draw (d) -- (g);
\draw (e) -- (f); \draw (f) -- (g);
\draw (u) -- (v); \draw (u) -- (w);
\draw (v) -- (x); \draw (x) -- (w);
\draw (a) -- (f);
\draw (a)  [fill=black] circle (\vr); \draw (b)  [fill=black] circle (\vr);
\draw (c)  [fill=black] circle (\vr); \draw (d)  [fill=black] circle (\vr);
\draw (e)  [fill=black] circle (\vr); \draw (f)  [fill=black] circle (\vr);
\draw (g)  [fill=black] circle (\vr); \draw (y)  [fill=black] circle (\vr);
\draw (u)  [fill=black] circle (\vr); \draw (v)  [fill=black] circle (\vr);
\draw (w)  [fill=black] circle (\vr); \draw (x)  [fill=black] circle (\vr);
\draw (6,.7) node {$0$}; \draw (6,-2.7) node {$[2,5]$};
\draw (-.7,0) node {$5$}; \draw (2.7,0) node {$5$};
\draw (-2,2.7) node {$3$}; \draw (4,2.7) node {$3$};
\draw (4,-2.7) node {$4$}; \draw (-2,-2.7) node {$4$};
\draw (8,.7) node {$1$}; \draw (8,-2.7) node {$1$};
\draw (10,.7) node {$2$}; \draw (10,-2.7) node {$2$};
\end{tikzpicture}
\end{center}
\caption{A $K_{2,3}$-amenable graph $G$}
\label{fig:K23}
\end{figure}

The graph $G$ in Figure~\ref{fig:K23} was constructed to have a weak partition that is $K_{2,3}$-amenable.
The partite sets of $K_{2,3}$ are as in the development above, $A=\{1,2\}$ and $B=\{3,4,5\}$.  For simplicity the
vertices of $G$ are labeled to indicate the subset of the weak partition that contains them.  For example, the
vertices labeled $1$ are in $V_1$ while the vertex labeled $[2,5]$ is the only member of $V_{[2,5]}$.
By Theorem~\ref{thm:completebiparitite} the Cartesian product $G \,\Box\, K_{2,3}$ is an EOD-graph.

\subsection{$G\,\Box\, C_{r}, r\in\{4,5\}$\label{5cycle}}

In this subsection we first define a type of weak partition of $V(G)$ that will
enable us to characterize those Cartesian products $G\,\Box\, C_{5}$ that are parallel EOD-graphs
with respect to $G$. To describe these weak partitions we need to modify Condition (C) as it
was stated in Section~\ref{sec:completefactor} and  add an additional condition.  The operations on the subscripts
in these new conditions are made modulo $5$.

\begin{itemize}
\item[(C$'$)] $\left\langle V_{i}\cup V_{i+1}\right\rangle $ is a
matching in $G$ for every $i\in [5]$,

\item[(D)] if $x\in V_{i}$, then $|N(x)\cap V_{i+2}|=1$ and
$|N(x)\cap V_{i-2}|=1$ for every $i\in [5]$.
\end{itemize}

Notice that the condition (C$'$) is weaker than (C). We say that $G$
is \emph{$C_{5}$-parallel amenable} if there exists a weak partition
$\{V_{0},V_{1},V_{2},V_{3},V_{4},V_{5}\}$ of $V(G)$ that satisfies
conditions (A), (B), (C$'$) and (D).

\begin{theorem}
\label{5cycle1} For any graph $G$, the Cartesian product $G\,\Box\, C_{5}$ is a parallel EOD-graph with respect to $G$
if and only if  $G$ is a $C_{5}$-parallel amenable graph.
\end{theorem}

\begin{proof} Assume first that $G$ is a $C_{5}$-parallel amenable graph
and let $\{V_{0},V_{1},V_{2},V_{3},V_{4},V_{5}\}$ be a weak partition of
$V(G)$ that satisfies conditions (A), (B), (C$'$) and (D). We define a subset $D$ of
$V(G\,\Box\, C_{5})$ by $D=\{(g,i) : g\in V_i\,\,\mbox{for}\,\, i\in[5]\}$.
Notice that $|D\cap\,^{g}C_{5}|=1$ for every $g\in V(G)-V_0$. We
will show that every vertex of $G\,\Box\, C_{5}$ is dominated by exactly
one vertex of $D$. Let $(g,j)$ be an arbitrary vertex of $G\,\Box\, C_{5}$.

Assume first that $g\in V_{i}$ for some $i\in [5]$.  If $j \in \{i-1,i+1\}$, then
$(g,j)$ is dominated by $(g,i)$. Moreover, $(g,j)$ is dominated only by $(g,i)$ in $D$,
since $(g,i)$ is the only vertex in $D\cap\,^{g}C_{5}$ and (B) and (C$'$)
hold.
If $j=i$, then $(g,j)$ is dominated by $(g',i)$, where $gg'$ is an
edge in $\left\langle V_{i}\right\rangle $ (notice that $g'$ exists
by (B)). Note that $(g,j)$ is dominated only by $(g',i)$ from $D$ by
(B) and the fact that $|D\cap\,^{g}C_{5}|=1$. It remains to consider $j=i+2$ and $j=i-2$.
Assume $j=i+2$; the case $j=i-2$ is similar.  By condition (D), $g$ has a unique
neighbor $x_{i+2}\in V_{i+2}$. By definition $(x_{i+2},i+2)\in D$, and thus
$(x_{i+2},i+2)$ dominates $(g,j)$.  As before, since $|D\cap\,^{g}C_{5}|=1$ and
since condition (D) holds, it follows that $(x_{i+2},i+2)$ is the only vertex of $D$
that dominates $(g,j)$.  Hence, if $g\in V_i$ for some $i\in[5]$, then $(g,j)$ is dominated
exactly once by $D$.
Finally, assume that $g\in V_{0}$. By condition (A), $g$ has a unique neighbor
$x_{j}\in V_{j}$.  This implies that $(x_{j},j)\in D$ and that $(x_j,j)$ is the only
vertex in $D$ that dominates $(g,j)$.  Consequently, $D$ is a parallel EOD-set of $G\,\Box\, C_{5}$.

Conversely, let $G\,\Box\, C_{5}$ be a parallel EOD-graph and let $D$ be a
parallel EOD-set with respect to $G$.  Let
$V_0,V_1,V_2,V_3,V_4,V_5$ be subsets of $V(G)$ defined as
follows.   If $|D\cap\,^gC_5|=0$, then $g\in V_{0}$. For $i\in [5]$,  $g\in\,V_{i}$ if and only if
$\{(g,i)\}=D\cap\,^{g}C_{5}$. By Lemma \ref{lem:atmosttwo}, $|D\cap\,^{g}C_{5}|\leq 1$ for
every $g\in V(G)$, and thus $\{V_{0},V_1,V_2,V_3,V_4,V_{5}\}$ is a weak partition
 of $V(G)$.  Note that only $V_0$ can be empty.

We will show that this weak partition satisfies conditions (A), (B), (C$'$)
and (D). If (A) does not hold, then there exists a vertex $(g,i)$
where $|D\cap\,^{g}C_{5}|=0$ and either $|N(g)\cap V_{i}|=0$ or
$|N(g)\cap V_{i}|>1$ for some  $i \in [5]$. In the first case
$(g,i)$ is not dominated by $D$ and in the second case $(g,i)$ is
dominated by at least two vertices, both contradicting the fact that
$D$ is a parallel EOD-set. Thus, (A) holds. If (B) is not satisfied,
then there exists $g\in V_{i}$, for some $i\in [5]$, such that
either $\deg _{\left\langle V_{i}\right\rangle }(g)=0$ or $\deg
_{\left\langle V_{i}\right\rangle }(g)>1$, which yields exactly the
same contradiction as for (A). Hence, (B) is true as well. If (C$'$)
does not hold, then there exist $g\in V_{i}$ and $g'\in V_{i+1}$ for
some $i\in [5]$, such that $gg'\in E(G)$. We infer that $(g,i+1)$ is
dominated twice, that is by $(g,i)$ and by $(g',i+1)$, which is not
possible. Finally, if (D) does not hold, then for some $i\in [5]$,
there exists $x\in V_{i}$  such that $|N(x)\cap V_{i+2}|\neq 1$
or $|N(x)\cap V_{i-2}|\neq 1$. Again we get that some vertex
is not dominated by $D$ (if $|N(x)\cap V_{i+2}|=0=|N(x)\cap V_{i-2}|$) or that
some vertex is dominated more than once by $D$
(if $|N(x)\cap V_{i+2}|>1$ or $|N(x)\cap V_{i-2}|>1$), which is not possible. This shows that  (D) is also
true, which completes the proof.
\end{proof}

While the complete characterization of  EOD Cartesian products where
one factor is $C_4\cong K_{2,2}$ was given in  Subsection~\ref{sec:completebipartite},
here we describe all $G$ such that $G\,\Box\, C_4$ is a parallel EOD-graph with respect to $G$.
For $C_{4}$ notice that computations on the subscripts are done modulo
$4$ in the set $[4]$, and in this case $i+2=i-2$.  Thus, we can
restate condition (D) as
\begin{itemize}
\item[\textrm{(D$'$)}] if $x\in V_{i}$, then $|N(x)\cap V_{i+2}|=1$ for
every $i\in [4]$.
\end{itemize}

\noindent We say that $G$ is $C_{4}$-parallel amenable if there exists a weak partition
$\{V_{0},V_{1},V_{2},V_{3},V_{4}\}$ of $V(G)$ that fulfills
conditions (A), (B), (C$'$) and (D$'$). The proof of the following
theorem follows the same lines as the proof of Theorem~\ref{5cycle1}
if we take into consideration computation modulo $4$ instead of
modulo $5$.

\begin{theorem} \label{c4}
For any graph $G$, the Cartesian product $G\,\Box\, C_{4}$ is a parallel EOD-graph with respect to $G$
if and only if  $G$ is a $C_{4}$-parallel amenable graph.
\end{theorem}

For $r\in \{4,5\}$ there exist many graphs $G$ which are not $C_r$-parallel amenable,
but for which $G\,\Box\, C_{4}$ is an EOD-graph (clearly $G\,\Box\, C_{4}$ is
not a parallel EOD-graph with respect to $G$ in this case). One of the smallest examples
is $P_3$, which is not $C_4$-parallel amenable, but $P_3\,\Box\, C_{4}$ is an
EOD-graph, even a parallel EOD-graph with respect to $C_4$.

\section{Conclusion}

As already mentioned, this method of defining weak partitions is most easily
implemented when one of the graphs has small diameter. Despite this fact,
there is no reason why one should not use it on graphs with larger diameter.
We illustrate this idea on a special case from the  class of cycles.

Our goal is to define a weak partition of a graph $G$ that consists of $V_0$
and a family of sets $V_A$ where $A$ is a subset of $[k]$ with
certain properties. We derive these properties from the second graph
in the product, which is $C_k$ now. Again we have two possibilities
for an edge from $\left\langle D\right\rangle$, where $D$ is an
EOD-set of $G\,\Box\, C_k$: either it projects to $C_k$ as an edge or as
a vertex. If it projects to an edge in $C_k$, then $A$ must contain
two consecutive elements $i$ and $i+1$. If an edge projects to a single
vertex $j\in V(C_k)$, then  $j\in A$ but neither $j+1$ nor $j-1$ is in $A$.
Moreover, two non-consecutive elements of $A$ must differ by at least
3 modulo  $k$, so that no vertex in the product is dominated
more than once.

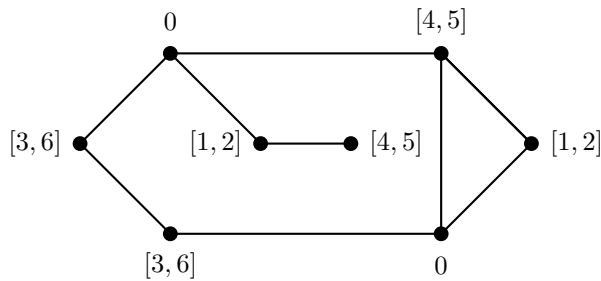
\begin{figure}[ht!]
\begin{center}
\begin{tikzpicture}[scale=0.6,style=thick]
\def\vr{4pt}
\path (6,0) coordinate (a); \path (4,-2) coordinate (b);
\path (-2,-2) coordinate (c); \path (0,0) coordinate (d);
\path (2,0) coordinate (e); \path (4,2) coordinate (f);
\path (-2,2) coordinate (g); \path (-4,0) coordinate (h);
\draw (a) -- (b); \draw (g) -- (h);
\draw (c) -- (h);
\draw (b) -- (c);
\draw (b) -- (f);
\draw (a) -- (f); \draw (f) -- (g);
\draw (d) -- (e); \draw (d) -- (g);
\draw (a) -- (f);
\draw (a)  [fill=black] circle (\vr); \draw (b)  [fill=black] circle (\vr);
\draw (c)  [fill=black] circle (\vr); \draw (d)  [fill=black] circle (\vr);
\draw (e)  [fill=black] circle (\vr); \draw (f)  [fill=black] circle (\vr);
\draw (g)  [fill=black] circle (\vr); \draw (h)  [fill=black] circle (\vr);
\draw (7,0) node {$[1,2]$}; 
\draw (-1,0) node {$[1,2]$}; \draw (3,0) node {$[4,5]$};
\draw (-2,2.7) node {$0$}; \draw (4,2.7) node {$[4,5]$};
\draw (4,-2.7) node {$0$}; \draw (-2,-2.7) node {$[3,6]$};
\draw (-5,0) node {$[3,6]$};
\end{tikzpicture}
\end{center}
\caption{A ``$C_{6}$-amenable'' graph $G$}
\label{fig:C6}
\end{figure}

In particular, for $C_6$ we obtain the following weak partition:
$$V_0,V_1,V_2,V_3,V_4,V_5,V_6,V_{[1,2]},V_{[2,3]},V_{[3,4]},V_{[4,5]},
V_{[5,6]},V_{[6,1]},V_{[1,4]},V_{[2,5]},V_{[3,6]}\,.$$
Clearly the size of the weak partition increases with $k$.
Together with this weak partition, several conditions are needed as well. For
instance we need a condition similar to (A) and (V) to care about
all vertices from $V_0$. As in the case of $K_{2,3}$-amenable
graphs, it seems to be hard to decide whether a graph $G$ is a
``$C_k$-amenable'' graph.  However, it is not difficult to construct
(small) examples of such graphs. An example of a ``$C_6$-amenable''
graph is given in Figure~\ref{fig:C6}.  The labeling follows the
prescription given for Figure~\ref{fig:K23}.


\nocite{*}
\bibliographystyle{abbrvnat}


\bibliography{EOD}

\end{document}